\documentclass[a4paper,12pt]{article}
 \usepackage[english]{babel}
  \usepackage{amsmath,amsfonts,amssymb,amsthm,bbm}
   \usepackage{graphics,epsfig,psfrag} 
   \usepackage{paralist,subfigure}

   \usepackage{hyperref} 

    \usepackage[active]{srcltx} 
    \usepackage{color,soul} 
    \usepackage[latin1]{inputenc}
    \usepackage[OT1]{fontenc}

\newtheorem{theorem}{Theorem}[section]
\newtheorem{corollary}[theorem]{Corollary}
\newtheorem{lemma}[theorem]{Lemma}
\newtheorem{proposition}[theorem]{Proposition}
\newtheorem{counter}{Counter-example}

\theoremstyle{definition}
\newtheorem{definition}[theorem]{Definition}
\newtheorem{remark}{Remark}

\def\N{\mathbb{N}}

\def\R{\mathbb{R}}

\let\vp=\varphi

\let\t=\tilde

\let\.=\cdot
\let\0=\emptyset

\def\1{\mathbbm{1}}

\def\L{\mathcal{L}}

\def\pe{principal eigenvalue}

\def\l{\lambda_1}
\def\mean#1{\left\langle#1\right\rangle}

\def\seq#1{(#1_n)_{n\in\N}}
\def\limn{\lim_{n\to\infty}}

\newenvironment{formula}[1]{\begin{equation}\label{#1}}
                       {\end{equation}\noindent}

\def\Fi#1{\begin{formula}{#1}}
\def\Ff{\end{formula}\noindent}

\title{\bf On the criticality and the \pe\ of almost periodic elliptic operators}
\author{Luca Rossi \\
\footnotesize{Istituto G.~Castelnuovo, Sapienza Universit\`a di Roma, Rome, Italy}
}

\date{}

\begin{document}

\maketitle


{\em In memory of Ha\"\i m Brezis,
whose work is an endless 
source of inspiration.}

\bigskip

\begin{abstract}
We review the notion and the properties of the generalised \pe\ for elliptic operators in unbounded domains,
and we relate it  with the criticality theory.
We focus on operators with almost periodic coefficients.
We present a Liouville-type result in dimension $N\leq2$. 
Next, we show with a counter-example that criticality
is not equivalent to the existence of an almost periodic
principal eigenvalue, even for self-adjoint operators.
Finally, we exhibit an almost periodic operator  which is subcritical but which admits 
 a critical limit operator.
 This is a manifestation of the instability character of the criticality
 property in the almost periodic setting.
\end{abstract}


\section{Introduction}

The {\em principal eigenvalue} encodes several fundamental properties of an elliptic operator, 
such as the validity of the maximum principle and the existence, uniqueness and stability of solutions.
The principal eigenvalue is classically defined as the eigenvalue which admits a positive eigenfunction.
For a uniformly elliptic operator with smooth coefficients, defined in a bounded smooth domain
under different types of boundary conditions (Dirichlet, Neumann, Robin), the \pe\ exists, 
it is unique and simple.
These facts can be shown using the Krein-Rutman theory, the extension of the Perron-Frobenius 
theory for operators on infinite dimensional spaces.
The crucial conditions required by the Krein-Rutman theory are the positivity and the compactness
of the resolvent of the operator, which are both consequences of the ellipticity as well as the
compactness of the spatial domain.
Indeed, the same theory can be applied in the case of periodic operators, that is, operators defined
on the torus, and provides one with the {\em periodic \pe}. 
In order to understand what happens beyond the periodic case, a first framework to investigate
is that of almost periodic (a.p.~in the sequel) functions, see the precise definition in Section~\ref{sec:results}.
A.p.~functions share the following property with continuous periodic functions: 
the set of their translations is relatively compact in $L^\infty$.
Despite this fact, one cannot cast the problem for operators with a.p.~coefficients
in a compact framework, unlike the periodic case.

In order to study the existence of an {\em a.p.~\pe} (i.e.~an eigenvalue with an a.p.~positive eigenfunction)
one should consider the eigenproblem in the whole
space and prove a-posteriori that the eigenfunction is a.p.
However, in such a framework the resolvent of the operator is not compact, making the Krein-Rutman theory unavailable.
As a matter of fact, both uniqueness and simplicity of the \pe~fail in general for operators in unbounded domains.
These questions have been addressed in \cite{BHR1,BHRossi,BR4} by introducing several
distinct notions of {\em generalised \pe}.

The differences between the periodic and the almost periodic settings are not only technical.
Indeed, the existence of the a.p.~\pe\ may fail for an operator with a.p.~coefficients.
A counter-example can be exhibited for an operator of the form
$\Delta-c(x)$ where $c(x)
= K (\cos (2\pi x) + \cos(2\pi\alpha x))$, $\alpha \notin \mathbb{Q}$ and $K$ is
large enough, which is a {\em quasiperiodic} function, a subclass of a.p.~functions.
This follows from a result of \cite{Bj0}  (see also~\cite{SoretsSpencer}) --~proved using 
large deviation techniques and the avalanche principle of \cite{Goldstein-Schlag}~--
and from Ruelle-Oseledec's theorem.
Quite surprisingly, if $K$ is small enough and $\alpha$ satisfies a certain diophantine condition,
then the above operator admits an a.p.~\pe, as shown by Kozlov \cite{Kozlov}.
In Counter-example \ref{ce:critical} below we present another example 
of a self-adjoint a.p.~operator in dimension $1$ which does not admit an a.p.~\pe.
This counter-example does not make use of probabilistic tools, but 
exploits an elementary construction performed in our previous paper~\cite{R2}, 
that we recall in Section \ref{sec:critical}. The 
counter-example further shows that 
the non-existence of the a.p.~\pe\ can also occur for a {\em critical} operator,
in the sense of Pinchover \cite{Pinch88}. The criticality theory is related, but not equivalent, 
to that of the generalised \pe, see Section~\ref{sec:results} below.  
Criticality implies in particular the uniqueness of positive solutions, up to a scalar~multiple.

It is not hard to see that for a self-adjoint operator in spatial dimension $N=1$,
the a.p.~\pe~is necessarily simple,
if it exists, see e.g.~\cite[Proposition~1.7]{NR3}. 
The same is true in dimension $N=2$ owing to a Liouville-type result by Nordmann~\cite{Nord}.
Combining the latter with a result by Pinchover~\cite{Pinch07}, one can prove a stronger property:
if an operator admits an a.p.~positive solution then it is critical if and only if 
the spatial dimension is $N\leq2$, see Corollary \ref{cor:N23} below.

We conclude this work with another counter-example to the 
existence of the a.p.~\pe, which in addition shows that an operator with a.p.~coefficients can be
{\em subcritical} but can have a critical {\em limit operator}, cf.~Counter-example~\ref{ce:subcritical}~below.

\bigskip\bigskip\bigskip


\section{Precise definitions and statements of the results}\label{sec:results}

We consider elliptic operators $\L$ of the form
\[\L u=\nabla\.(A(x)\nabla u)+b(x)\.\nabla u+c(x)u,\qquad x\in\R^N,\]
where $A$ is a matrix field, $b$ is a vector field and $c$ is a scalar function.
We always assume that the entries of $A,b$ and the function $c$ are real valued, measurable and~bounded, and,
in addition, that $A$ is smooth and it is symmetric and uniformly elliptic.

One cannot define the principal eigenvalue of $-\L$ in a classical way through the Krein-Rutman theory, 
because the domain $\R^N$ is unbounded.
In collaboration with H.~Berestycki and F.~Hamel,
motivated by the study of reaction-diffusion equations in heterogeneous media,
we considered in \cite{BHRossi} the following notion of {\em generalised \pe}:
\[\lambda_1:=\sup\{\lambda\in\R\ :\ \exists\phi>0, -\L\phi\geq\lambda\phi\text{ on }\R^N\}.\]
The functions $\phi$ in the above formula (and in general the functions involved in differential equalities/inequalities)
are understood to belong to $W^{2,N}_{loc}(\R^N)$ (and to satisfy the inequality a.e.).
We point out that no global integrability condition is required on~$\phi$.
The above definition coincides with the one introduced by
Berestycki-Nirenberg-Varadhan \cite{BNV} in the case of linear elliptic operators defined on
bounded non-smooth domains.
%
The authors showed that, up to giving a relaxed meaning to the Dirichlet boundary condition,
such notion enjoys all the fundamental properties of the classical one. 
Some related notions and partial results had been previously given in \cite{NP}.
Instead, in the whole space (or more in general in unbounded domains), 
the quantity $\lambda_1$ is not sufficient to encode all the classical properties
of the \pe.
In fact, another notion, denoted by $\l'$,
has been employed in \cite{BHR1,BHRossi} to give a complete description 
of solutions of the elliptic problem, and a third quantity, $\l''$, has been introduced in \cite{BR4}.
The point is that the eigenvalue admitting a positive eigenfunction is not unique,
as in the case of a bounded domain, but actually the whole half line $(-\infty,\l]$ does, see \cite[Theorem~1.4]{BR4}.
For instance, for a periodic operator $-\L$ which is not self-adjoint, 
the periodic \pe~is typically smaller than $\l$, and actually coincides with $\l'$ and $\l''$, 
see Section \ref{sec:pe} below.

We are concerned in particular with operators with almost periodic coefficients.

\begin{definition}\label{def:ap}
A function $f:\R^N\to\R$ is said to be {\em almost periodic} (a.p.) if from any sequence $(x_n)_{n\in\N}$ in $\R^N$
one can extract a subsequence $(x_{n_k})_{k\in\N}$
such that $(f(x+x_{n_k}))_{k\in\N}$ converges uniformly in $x\in\R^N$.
\end{definition}

We say that the elliptic operator $\L$ is a.p.~if its coefficients, i.e.~the components of the matrix field $A$, 
of the vector field $b$ and the function $c$, are all a.p.~functions.
\\


The {\em criticality theory} has been developed by 
Pinchover \cite{Pinch88} in order to study 
the existence, uniqueness and stability of positive solutions of elliptic equations
beyond the case of operators with compact resolvent, and also by \cite{Murata} in the case of
Schr\"odinger operators. 

\begin{definition}\label{def:criticality}
An elliptic operator $\L$ is 
\begin{itemize}
\item {\em supercritical} if the set of positive solutions of $\L u=0$ is empty;

\item {\em critical} if the set of positive functions satisfying $\L u\leq0$ is one-dimensional;

\item {\em subcritical} in the other cases.
\end{itemize}
\end{definition}

The above definition differs from the original one of Pinchover \cite{Pinch88}, 
which is expressed in terms of the existence a positive Green function or of a ground state, in the sense
of Agmon~\cite{A1}. For a comprehensive treatment of the subject, from both the PDE and the probability point of view,
we refer to the monograph by Pinsky~\cite{PiBook}.
The above equivalent definition is given in \cite{Pinch07}, cf.~also~\cite{A1,PinchLandis} or \cite[Theorems~3.4 and~3.9]{PiBook}.
A previous equivalent definition in the case of Schr\"odinger operators had been given by Simon \cite{Simon}.
Criticality can be related to the generalised \pe:
\[\l<0\ \iff \ \L \text{ is supercritical},\]
\[\l>0\ \implies \ \L \text{ is subcritical},\]
\[\l=0\ \implies \ \L \text{ is either critical or subcritical}.\]
Hence the criticality is not completely characterised by $\l$.

For a critical operator, the set of positive solutions of $\L u=0$ coincides with the 
set of positive supersolutions, hence it is also one-dimensional, by definition.
In dimension $N=1$, the viceversa holds true, namely, an operator is critical if and only if 
the set of positive solutions of $\L u=0$ is one-dimensional,
see \cite[Proposition~5.1.3]{PiBook} (or \cite[Appendix~1]{Murata} for the self-adjoint case).

It is well known that the Laplace operator in $\R^N$ is critical if and only if ${N\leq2}$.
This result has been extended by Damanik-Killip-Simon \cite[Theorem 5]{DKS} to the operator $\L=\Delta+c(x)$.
A related result is given by Pinsky \cite[Theorem 8.2.7]{PiBook}, proved using probabilistic techniques,
that asserts that for a periodic operator $\L$, not necessarily self-adjoint, 
$\L+\lambda_1$ is critical if and only if $N\leq2$. Using the results of \cite{BCN2,Nord,Pinch07}
one can obtain still another extension to non-periodic operators.
 
\begin{proposition}\label{pro:N23}
	Let $\L$ be a self-adjoint operator with regular coefficients. There holds:
	\begin{enumerate}[$(i)$]
		\item If $N\leq 2$ and the equation $\L=0$ admits a positive bounded solution $\varphi$ then
		$\L$ is critical.
		\item If $N\geq 3$ and the equation $\L=0$ admits a solution $\varphi$ satisfying $\inf\varphi>0$ then
				$\L$ is subcritical.
	\end{enumerate}
\end{proposition}

\begin{proof} $(i)$ \cite[Theorem~2]{Nord} implies that $\varphi$ is a
	{\em generalised principal eigenfunction} of~$\L$, i.e.~$\l=0$.
	Let $\psi>0$ satisfy $\L\psi\leq0$. Call $u:=-\psi+k\vp$ with $k>0$ large enough so that $\sup u>0$.
Thus $u$ is a subsolution of $\L=0$ and it is bounded from above.
It~then follows again from \cite[Theorem~2]{Nord} that $u$ is a multiple of~$\vp$, whence so is $\psi$.
This shows that $\L$ is critical.

$(ii)$ Assume by contradiction that $\L$ is critical. Then applying \cite[Theorem~1.7]{Pinch07} (with $\psi\equiv1$)
one deduces that the operator $\nabla\.(A\nabla)$ is critical too. But this is not the case in dimension $3$ or larger, 
by 
\cite{BCN2,GG,Barlow}.
\end{proof}

An a.p.~function is necessarily bounded. Moreover if it is a positive solution of 
an elliptic equation, then it must have a positive infimum, because
otherwise it would converge locally uniformly (hence uniformly) to 0 around minimising sequences,
which is impossible.
One then derives from Proposition \ref{pro:N23} the following.

\begin{corollary}\label{cor:N23}
	Let $\L$ be a self-adjoint operator with regular coefficients and assume that the equation $\L u=0$ admits a positive a.p.~solution $\varphi$.
	 Then $\L$ is critical if and only if $N\leq 2$.
\end{corollary}

Corollary~\ref{cor:N23} extends the above mentioned result \cite[Theorem 8.2.7]{PiBook} in the case of self-adjoint operators,
because for a self-adjoint periodic operator, the generalised \pe\ $\lambda_1$ reduces to the periodic \pe, 
see \cite[Proposition~6.6]{BHRossi}, that is, $\L+\lambda_1$ admits a positive periodic~solution.
One can wonder whether the reciprocal of Corollary \ref{cor:N23} holds true for an a.p.~operator, 
namely, if in low dimension criticality is equivalent to the existence of an a.p.~\pe. 
The following gives a negative answer.

\begin{counter}\label{ce:critical}
 There exists a self-adjoint a.p.~operator $\L$ in dimension ${N=1}$ 
 which is critical but does not admit any positive a.p.~eigenfunction.
\end{counter}

The proof of Counter-example \ref{ce:critical}
is given in Section \ref{sec:critical}. It makes use of a construction carried out in our previous paper \cite{R2}, 
that we will reclaim. The operator of the counter-example possesses a ground state with an almost exponential decay,
which is related to the phenomenon of {\em Anderson localization}, see \cite{BouKen}.

We conclude with a last result that shows, on the one hand, that self-adjoint a.p.~operators with $\l=0$
may not be critical, even in dimension $1$, and, on the other hand, that criticality is not invariant
in the passage to limit operators, which, loosely speaking, are the $\omega$-limits
of the original operator.

\begin{definition}\label{def:limitop}
A {\em limit operator} associated with $\L$ is an operator of the form
\[\L^*u=\nabla\.(A^*(x)\nabla u)+b^*(x)\.\nabla u+c^*(x)u,\]
where $A^*(x), b^*(x), c^*(x)$ are the pointwise limits of $A(x+x_n), b(x+x_n), c(x+x_n)$
respectively, for a given sequence $\seq{x}$ in $\R^N$.
\end{definition}

\begin{counter}\label{ce:subcritical}
 There exists a self-adjoint a.p.~operator $\L$ in dimension ${N=1}$ which is subcritical, it has $\l=0$,
 but it admits a limit operator $\L^*$ which is critical.
\end{counter}

\begin{remark}\label{rk:subcritical}
 For an a.p.~operator $\L$, the convergences of $A(x+x_n)$, $b(x+x_n)$, ${c(x+x_n)}$ in the definition of 
 the limit operator $\L^*$ are actually uniform in $x\in\R^N$ (up to subsequences). It follows that
 $A^*(x-x_n)$, $b^*(x-x_n)$, ${c^*(x-x_n)}$ converge (up to subsequences) to $A(x), b(x), c(x)$,
 hence $\L$ is a limit operator associated with $\L^*$ (which is also a.p.).
 In particular, we can rephrase Counter-example \ref{ce:subcritical} by saying that
 there exists a self-adjoint a.p.~operator in dimension ${N=1}$ which is critical, 
  but it admits a limit operator which is subcritical.
\end{remark}

The operator $\L$ in Counter-example \ref{ce:subcritical} exhibits three striking differences with respect
to periodic operators: \\
\indent 1) $\L$ does not admit an a.p.~(or merely a bounded) positive eigenfunction, owing to Proposition~\ref{pro:N23};\\
\indent 2) for a periodic operator $\t\L$, $\t\L+\l$ is always critical
	in dimension $N\leq2$, by~\cite[Theorem~8.2.6]{PiBook}, while $\L$ is not;\\
\indent 3) a periodic operator shares the same criticality character with all its limit operators
(which are just translations of the former), while $\L$ does not.


\section{Other notions of generalised \pe}\label{sec:pe}

%
%
%
%
%
%
%
%
%

Besides the generalised principal eigenvalue $\l$ 
introduced in \cite{BNV}, two other notions were employed in \cite{BHR1,BHRossi,BR4} to characterise 
the existence, uniqueness and the positivity of solutions of elliptic equations on 
unbounded domains, with application to the study of the Fisher-KPP equation in heterogeneous media. 
In the case where the domain is the whole space, they reduce to:
\[\lambda_1':=\inf\{\lambda\in\R\ :\ \exists\phi>0,\ \sup\phi<+\infty,
 -\L\phi\leq\lambda\phi\text{ on }\R^N\},\]
and 
\[\lambda_1'':=\sup\{\lambda\in\R\ :\ \exists\phi,\ \inf\phi>0,\
-\L\phi\geq\lambda\phi\text{ on }\R^N\}.\]
One of the key properties of these other notions is that the operator $\L$ fulfils the {\em Maximum Principle} 
if $\l''>0$ and only if $\l'\geq0$, c.f.~\cite[Theorem~1.6]{BR4}.

One has that the inequalities $\l''\leq\l'\leq\l$ hold in general, while for a self-adjoint operator
it holds that $\l'=\l$, see \cite[Theorem~1.7]{BR4}.
It is further conjectured in \cite{BR4} that $\l''=\l'$.
Using a probabilistic approach connecting elliptic PDEs and branching processes,
Maillard-Tough~\cite{MT-proba} have recently proved this conjecture in the case of self-adjoint operators
with H\"older continuous coefficients, and disproved it with a counter-example for non-self-adjoint operators.

For a periodic operator, both quantities $\l'$ and $\l''$  coincide with the periodic \pe,
that is, the unique eigenvalue admitting a positive periodic eigenfunction, while $\l$ 
is in general greater than the latter. Berestycki and Nadin showed
that, for an operator with a.p. (and even {\em uniquely
ergodic}), H\"older continuous~coefficients, not necessarily self-adjoint, one has $\l'=\l''$.
This result is derived in \cite{BNadin-multi} using the homogenization techniques of Lions-Souganidis~\cite{LS2}.
Indeed, the study of eigenproblems can be recast in the homogenization framework in
terms of the existence of {\em approximate correctors} for a transformed operator.

\section{Non-existence of an a.p.~\pe\ in the critical case}\label{sec:critical}

In \cite{R2}, an a.p.~function $b$ has been constructed in order to exhibit a counter-example to 
a Liouville-type result for elliptic operators.
We use the same function $b$ to construct the operator $\L$ of Counter-example~\ref{ce:critical}.
Let us recall its definition: first, we~set 
\[\sigma(x):=\begin{cases}
-1 & \text{ if } x\in[-1,0)\\[5pt]
1 & \text{ if } x\in[0,1],
\end{cases}\]
then, by iteration,
\Fi{iter}
\forall n\in\N,\quad
\sigma(x):=\begin{cases}
\displaystyle \sigma(x+2\.3^n)-\frac1{(n+1)^2} & \text{ if } x\in[-3^{n+1},-3^n)\\[12pt]
\displaystyle \sigma(x-2\.3^n)+\frac1{(n+1)^2} & \text{ if } x\in(3^n,3^{n+1}].
     \end{cases}
\Ff
Finally, in order to have a regular function, we define
\[b(x):=\sigma(x)\sin^4(\pi x).\]
The function $b$ is depicted in Figure \ref{fig:b}. It is of class $C^3$, 
odd and {\em limit periodic}, that is, it is the uniform limit
of a sequence of continuous periodic functions. In this case, such a sequence is given by $\seq{b}$ with
$b_n\equiv b$ on $[-3^n,3^n]$ and then extended by periodicity.
Limit periodic functions are a subclass of the a.p.~functions because, as it
is easily seen from Definition \ref{def:ap}, the space of
a.p.~functions is closed with respect to the $L^\infty$ norm (see
e.g.~\cite{a.p.,Fink}).
Moreover, 
one can show the following inequality:
\Fi{<b>}
\forall x\in\R,\qquad
\int_0^x b(t)dt\geq K\frac{|x|}{(\log_3|x|+1)^2}-K,
\Ff
for some $K>0$, see \cite[Proposition 3.3]{R2}.
\begin{figure}[ht]
 \centering
 \includegraphics[width=.9\linewidth]{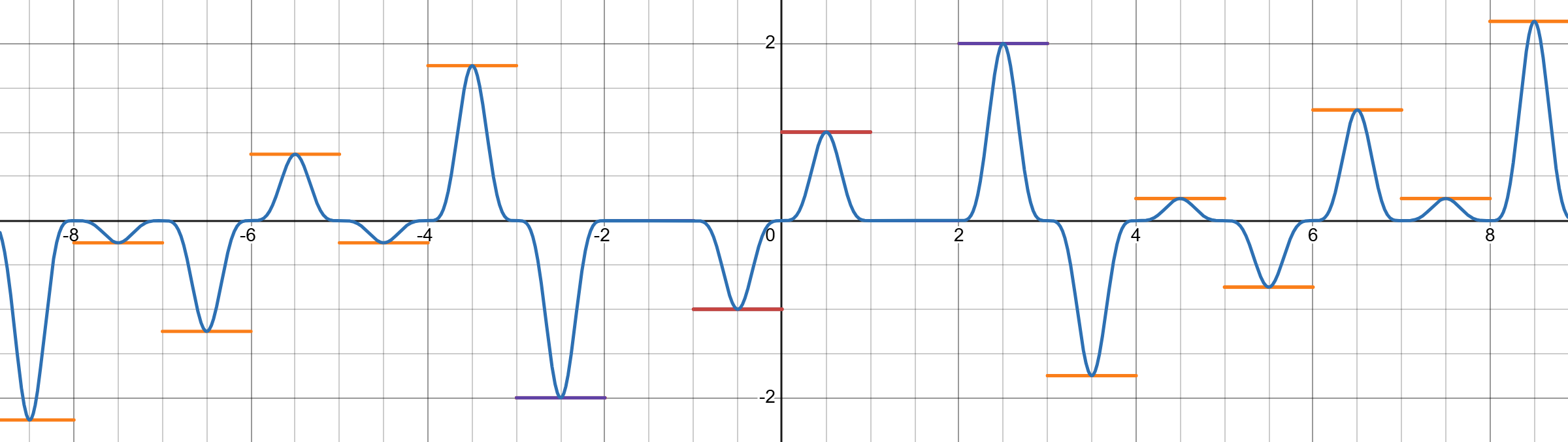}
 \caption{The graphs of $\sigma$ and $b$.}
 \label{fig:b}
 \end{figure}

\begin{proof}[Proof of Counter-example~\ref{ce:critical}]
Consider the following operator:
$$L u:=u''-2b(x)u',\qquad x\in\R,$$
with $b$ defined as above.
The function
$$v(x):=e^{-\int_0^xb(t)dt}u(x)$$
satisfies
$\L v= e^{-\int_0^xb(t)dt}\,Lu$, with
$$\L v:=v''+(b'-b^2)v.$$
The coefficient $b'-b^2$ is a.p., because $b^2$ is a.p., and also $b'$ is, being uniformly continuous,
and it is known that the derivative of an a.p.~function is a.p.~if and only if it is uniformly continuous,
see e.~g.\ \cite[property V]{a.p.}.

We now investigate the criticality of $\L$ and the sign of its \pe\ $\l$.
The function $$\vp(x):=e^{-\int_0^xb(t)dt}$$ 
satisfies $\L\vp=0$.
It follows from \eqref{<b>} that $\vp$ decays to 0 at $\pm\infty$, hence it is bounded.
Recalling the definition of $\l$ and of $\l'$ in Section \ref{sec:pe} one derives $\l\geq0\geq\l'$. But then,
since $\l=\l'$ by \cite[Theorem 1.7(i)]{BR4}, one eventually
deduces that $\l=0$.
On the other hand, in dimension 1, the existence of a bounded positive eigenfunction 
implies  that $\L$ is critical, as it follows from \cite[Theorem~A.9]{Murata} or also
\cite[Theorem~5]{DKS}, \cite{Pinch07}, or even from Proposition \ref{pro:N23} above.
In turn, criticality entails that the space of positive solutions to 
$\L=0$ has dimension 1.
Therefore, $\L$ does not admit an a.p.~positive eigenfunction, since the 
only candidate,~$\vp$, decays to 0 and thus it cannot be a.p. 

Notice that, instead, $L$ is critical and the unique positive eigenfunction 
is the constant one, which is a.p.
\end{proof}


\section{The Subcritical case}\label{sec:subcritical}

This section is devoted to the construction of Counter-example \ref{ce:subcritical}.
We start again from the function $b$ of \cite{R2}, reclaimed in the previous section.
We define
$$L u:=u''+2b(x)u',\qquad x\in\R.$$
Similarly to before, we transform $L$ into a 
self-adjoint operator by considering
$$v(x)=e^{\int_0^xb(t)dt}u(x),$$
namely $\L v= e^{\int_0^xb(t)dt}\,Lu$ with 
\[\L v:=v''-(b'+b^2)v.\]
As noticed in Section~\ref{sec:critical}, this is an a.p.~operator.
%
We first study the generalised \pe\ and the criticality character of~$\L$.

\begin{proposition}
The operator $\L$ is subcritical and satisfies $\l=0$.
\end{proposition}

\begin{proof}
Let $L$, $\L$ be the above operators.
The solutions to $Lu=0$ satisfy 
\[u'(x)=Ce^{-2{\int_0^xb(t)dt}},\]
for some $C\in\R$, hence they are bounded by \eqref{<b>}.
Then, letting $u$ be as above with $C\neq0$, the functions 
$$u_1\equiv1,\qquad u_2\equiv u+\|u\|_\infty$$
are two linearly independent positive solutions of $L=0$. This means that
$L$ is subcritical and has $\l(-L)\geq0$
(here and in the sequel we are stressing the operator to
which the \pe\ is referred to).
Considering the functions $e^{\int_0^xb(t)dt}u_{1,2}(x)$ one infers that also $\L$ is subcritical with $\l(-\L)\geq0$.

It remains to show that $\l(-\L)=0$, i.e.~that $\l(-\L)\leq 0$.
We will make use of the \pe\ $\l'(-\L)$ defined in Section \ref{sec:pe}.
We indeed know from \cite[Theorem 1.7(i)]{BR4} that $\l(-\L)=\l'(-\L)$.
Notice that the solutions to $\L=0$ grow like $e^{\int_0^x b}$, which is 
unbounded, and thus they cannot be used in the definition of $\l'(-\L)$. 
However, we will show that, for any $\lambda>0$, $(\L +\lambda )v=0$ admits 
a positive bounded subsolution. Fix $\lambda>0$. Define
the following function: 
$$\psi(x):=e^{\int_0^x[b(t)-\gamma(t)]dt},$$ 
where $\gamma\in C^1(\R)$ will be chosen later.
We compute
$$\L\psi=[(b-\gamma)'+(b-\gamma)^2]\psi-(b'+b^2)\psi
=(\gamma^2-2b\gamma-\gamma')\psi.$$
We choose $\gamma$ small enough (in $C^1$ norm) so that $\gamma^2-2b\gamma-\gamma'\geq-\lambda$ in $\R$,
and in addition $\gamma(-\infty)<0<\gamma(+\infty)$. It follows that 
$-\L\psi\leq\lambda\psi$.
In order to see that $\psi$ is bounded, we use the fact that $b$ has zero average.
We recall that for an a.p.~function~$f$, the average
$$\mean{f}:=\lim_{h\to\pm\infty}\frac1h\int_{x_0}^{x_0+h}f(t)dt$$
exists uniformly with respect to $x_0$ (and it is independent of $x_0$).
Then, being $b$ odd and a.p., it has zero average, since
$\int_{-h}^{h} b=0$ for all $h\in\R$.
As a consequence,
$$\lim_{x\to\pm\infty}\frac1x\int_0^x(b-\gamma)=-\gamma(\pm\infty),$$
which immediately implies that $\psi\to0$ as $x\to\pm\infty$.
We eventually infer, from the definition, that $\l'(-\L)\leq\lambda$, and 
therefore $\l'(-\L)\leq0$ by the arbitrariness of $\lambda>0$.
We have shown that $\l(-\L)=\l'(-\L)\leq0$, as desired.
\end{proof}



It remains to show that the subcritical operator $\L$ 
admits a limit operator which is critical. 
%
Owing to Proposition~\ref{pro:N23}, it is sufficient to exhibit a limit operator which admits
a positive bounded eigenfunction with eigenvalue $0$ (see also 
\cite[Proposition~1.7]{NR3} and the characterization of critical operators in \cite{Murata}).
We now improve these results by relaxing the boundedness of the eigenfunction,
using the Harnack inequality. 
\begin{lemma}\label{lem:critical}
Let $\L$ be a self-adjoint operator on $\R$ and assume that there exist $\lambda\in\R$ and a positive
function $\vp$ satisfying $-\L\vp=\lambda\vp$ and
$$\liminf_{x\to-\infty}\vp(x)<+\infty\qquad\text{and}\qquad
\liminf_{x\to+\infty}\vp(x)<+\infty.$$
Then necessarily $\lambda=\l$ 
and moreover $\L+\lambda$ is critical.
\end{lemma}

\begin{proof}
Let $\psi$ be another eigenfunction of $-\L$, not identically equal to $\vp$, associated with the same 
eigenvalue $\lambda$ as~$\vp$. We can assume without loss
of generality that $\psi(0)\neq0$. Let us normalize $\vp,\psi$ 
by $\vp(0)=\psi(0)=1$.
Using the equation $(\L+\lambda)=0$ satisfied by both $\vp$ and $\psi$ we get, for all 
$x\in\R$,
$$\int_0^x (A\vp')'\psi=\int_0^x(A\psi')'\vp,$$
from which we deduce
$$A(x)(\vp'\psi-\psi'\vp)(x)=A(0)(\vp'-\psi')(0),$$
or equivalently
\Fi{psi/vp}
(\psi/\vp)'(x)=\frac{A(0)(\psi'-\vp')(0)}{A(x)\vp^2(x)}.
\Ff
Since $\psi\not\equiv\vp$, we know that $(\psi'-\vp')(0)\neq0$.
Up to reflecting the functions $\vp$, $\psi$ and the coefficients of $\L$, it is not restrictive to 
assume that $(\psi'-\vp')(0)>0$ (notice that the reflection does not affect 
the hypothesis on $\vp$).
We then infer from \eqref{psi/vp} that $\psi/\vp$ is increasing on $\R$.   
Next, by hypothesis, there exists a sequence $\seq{x}$ in $\R$ 
such that 
$$x_n\to-\infty,\qquad\limn\vp(x_n)\in[0,+\infty).
$$
We can assume that $\seq{x}$ satisfies (up to subsequences) $x_{n+1}\leq x_n-1$ 
for $n\in\N$.
We now use the Harnack inequality. It implies the existence of a 
positive constant $C$ such that 
$$\forall n\in\N,\ x\in[x_n-1,x_n+1],\quad
\vp(x)\leq C.$$
This fact, together with \eqref{psi/vp}, yields
$$\forall n\in\N,\ x\in[x_n-1,x_n+1],\quad
(\psi/\vp)'(x)\geq\frac{A(0)(\psi'-\vp')(0)}{C^2\sup A},$$
from which we deduce
\[\begin{split}
(\psi/\vp)(x_{n+1}) &\leq(\psi/\vp)(x_n-1)\leq
(\psi/\vp)(x_n)-\frac{A(0)(\psi'-\vp')(0)}{C^2\sup A}\\
&\leq\cdots\leq (\psi/\vp)(x_1)-n\frac{A(0)(\psi'-\vp')(0)}{C^2\sup A}.
\end{split}\]
We eventually obtain that $(\psi/\vp)(x_n)\to-\infty$ as $n\to\infty$.
This shows that $\vp$ is the unique positive
eigenfunction of $-\L$ with eigenvalue $\lambda$, up to a scalar multiple, 
which, according to \cite[Proposition~5.1.3]{PiBook}, implies that 
$\L+\lambda$ is critical, whence in particular $\lambda=\l$.
\end{proof}

In order to conclude the proof of Counter-example \ref{ce:subcritical},
it remains to find a limit operator associated with $\L$ that fulfils the hypothesis of Lemma \ref{lem:critical}.
We claim that the function $b$ (defined in Section \ref{sec:critical}) satisfies the following:
\Fi{kn>}
\forall k,n\in\N,\ k\leq n,\qquad\int_{3^n-3^k}^{3^n}b(x)dx\geq0.
\Ff
Let us first show that \eqref{kn>} holds in the case $n=k$, i.e.~that $\int_{0}^{3^n}b(x)dx\geq0$.
This inequality holds fo $n=0$. Assuming that it holds for some $n\in\N$, using \eqref{iter} we get
\[
\begin{split}
\int_{0}^{3^{n+1}}b(x)dx &
=\int_{0}^{3^{n}}b(x)dx+\int_{3^n}^{3^{n+1}}[\sigma(x-2\.3^n)+(n+1)^{-2}]\sin^4(\pi x)dx\\
&\geq\int_{-3^n}^{3^n}b(x)dx,
\end{split}
\]
and the latter integral is 0 because $b$ is odd.
Hence \eqref{kn>} holds for $n=k$. We now fix $k\in\R$
and proceed by iteration on $n$.
Using 
\eqref{iter} we derive
$$\int_{3^{n+1}-3^k}^{3^{n+1}}b(x)dx=\int_{3^{n+1}-3^k}^{3^{n+1}}[\sigma(x-2\.3^n)+(n+1)^{-2}]\sin^4(\pi x)dx\geq
\int_{3^n-3^k}^{3^n}b(x)dx,$$
and thus, if \eqref{kn>} holds for some $n$, then it holds true for $n+1$.
This proves \eqref{kn>}.

Next, using again \eqref{iter}, we deduce, for $k\leq n$,
\[\begin{split}
\int_{3^n}^{3^n+3^k}b(x)dx &=\int_{3^n}^{3^n+3^k}[\sigma(x-2\.3^n)+(n+1)^{-2}]\sin^4(\pi x)dx\\
&=\int_{-3^n}^{-3^n+3^k}b(x)dx+
\frac{C 3^k}{(n+1)^2},
\end{split}\]
where $C=\int_0^1\sin^4(\pi x)dx$.
Then, since $b$ is odd, we obtain
$$\int_{3^n}^{3^n+3^k}b(x)dx=\int_{3^n}^{3^n-3^k}b(y)dy+
\frac{C3^k}{(n+1)^2},$$
and therefore, by \eqref{kn>}, we eventually derive
\Fi{kn<}
\forall k,n\in\N,\ k\leq n,\quad\int_{3^n}^{3^n+3^k}b(x)dx
\leq \frac{C3^k}{(n+1)^2}.
\Ff

Let us consider the sequence $(b_n)_{n\in\N}$ of translations of $b$ defined by
$$b_n(x):=b(x+3^n).$$
By almost periodicity, $(b_n)_{n\in\N}$ and $(b'_n)_{n\in\N}$ converge uniformly
(up to subsequences) to some function 
$b_\infty$ and its derivative $b_\infty'$, respectively.
It follows from \eqref{kn>} that
$$\forall 
k\in\N,\quad\int_{-3^k}^0b_\infty(x)dx=\limn\int_{-3^k}^0b_n(x)dx\geq0,$$
while, from \eqref{kn<},
$$\forall 
k\in\N,\quad\int_{0}^{3^k}b_\infty(x)dx=\limn\int_{0}^{3^k}b_n(x)dx\leq
\limn\frac{C3^k}{(n+1)^2}=0.$$
We finally consider the limit operator
$$\L^*u:=u''-(b_\infty'+b_\infty^2) u.$$
The function
$$\vp(x):=e^{\int_0^xb_\infty(t)dt}$$
satisfies $\L^*\vp=0$ and, in addition,
$$\forall k\in\N,\quad
\vp(\pm 3^k)=e^{\int_0^{\pm3^k}b_\infty(t)dt}\leq1.$$
We can therefore apply Lemma \ref{lem:critical} and infer that $\L^*$ is 
critical.


\end{document}